\documentclass{amsart}

\usepackage[latin1]{inputenc}
\usepackage{amsfonts}
\usepackage{amsmath}
\usepackage{amsthm}
\usepackage{amssymb}
\usepackage{latexsym}
\usepackage{enumerate}
\usepackage{mathtools}
\usepackage{dsfont}
\usepackage{xcolor}

\newtheorem{theorem}{Theorem}[section]
\newtheorem{lemma}[theorem]{Lemma}
\newtheorem{proposition}[theorem]{Proposition}

\newtheorem{questions}[theorem]{Questions}
\newtheorem{question}[theorem]{Question}

\newtheorem{definition}[theorem]{Definition}

\numberwithin{equation}{section}

\DeclareMathOperator{\supp}{supp}
\DeclareMathOperator{\ran}{ran}

\newcommand{\cc}{\mathfrak{c}}
\newcommand{\N}{\mathbb{N}}

\newcommand{\Q}{\mathbb{Q}}
\newcommand{\R}{\mathbb{R}}

\newcommand{\PP}{\mathbb{P}}

\newcommand{\F}{\mathcal{F}}
\newcommand{\G}{\mathcal{G}}
\newcommand{\A}{\mathcal{A}}
\newcommand{\B}{\mathcal{B}}

\newcommand{\CC}{\mathcal{C}}
\newcommand{\deff}{\vcentcolon =}

\newcommand{\ang}[1]{\langle #1 \rangle}
\newcommand{\simi}[1]{\overset{\scriptscriptstyle{ #1 }}\sim}

\newcommand{\norm}[1]{\| #1 \|}

\newcommand{\mc}[1]{\mathcal{{#1}}}
\newcommand{\mr}[1]{\mathrm{{#1}}}

\begin{document}

\author{Piotr Koszmider}
\address{Institute of Mathematics of the Polish Academy of Sciences,
ul.  \'Sniadeckich 8,  00-656 Warszawa, Poland}
\email{\texttt{piotr.math@proton.me}}

\author{Ma\l gorzata Rojek}
\address{Warsaw Doctoral School of Mathematics and Computer Science,
ul. S. Banacha 2c,
02-097 Warszawa, Poland}
\email{\texttt{m.rojek3@uw.edu.pl}}

\thanks{The  authors were  partially supported by the NCN (National Science
Centre, Poland) research grant no.\ 2020/37/B/ST1/02613.}
\thanks{The  first-named author would like to thank Jes\'us Castillo for introducing him
to the topic which is the subject of this paper during his visit to Warsaw in the fall of 2023.}
\thanks{The authors would like to thank  Bill Johnson for suggesting literature
concerning finite dimensional Banach spaces which led us to the discovery of the paper \cite{bourgain}
which is crucial in Section~4.}


\title[Injective and automorphic properties of $\ell_\infty/c_0$]{Almost disjoint families and some 
automorphic and injective properties of $\ell_\infty/c_0$}

\begin{abstract} Answering questions of A. Avil\'es, F. Cabello S\'anchez, J.  Castillo, M. Gonz\'alez and Y.  Moreno
we show that the following statements are independent of the usual {\sf ZFC} axioms with arbitrarily large continuum:
for every (some)  $\omega<\kappa<2^\omega$
\begin{enumerate}
\item any linear bounded operator $T: c_0(\kappa)\rightarrow\ell_\infty/c_0$
extends to any superspace of  $c_0(\kappa)$.
\item any isomorphism between any two copies of $c_0(\kappa)$ inside $\ell_\infty/c_0$
extends to an automorphism of $\ell_\infty/c_0$.
\end{enumerate} 
This contrasts with  Boolean, 
Banach algebraic or isometric levels, where the objects known as Hausdorff gap and Luzin gap 
witness the failure in {\sf ZFC} of the corresponding properties for the corresponding structures 
already at the first
uncountable cardinal $\kappa=\omega_1$.

In particular, consistently, any two pairwise disjoint families in $\wp(\N)/Fin$ of the same cardinality $\omega<\kappa<2^\omega$
can be mapped onto each other by a linear automorphism of $\ell_\infty/c_0$ regardless
of their different combinatorial, algebraic or topological positions in $\wp(\N)/Fin$.

Our positive consistency results use a restricted version of Martin's axiom for a partial order
that adds an infinite block diagonal matrix of an operator on $\ell_\infty$ which induces an operator on $\ell_\infty/c_0$.
The construction of
its finite blocks relies on a lemma of Bourgain and Tzafriri on finite dimensional Banach spaces.

Our negative consistency results rely on an analysis of almost disjoint families of $\N$,
the embeddings of $c_0(\kappa)$ into $\ell_\infty/c_0$ they induce and
their extensions to $\ell_\infty^c(\kappa)$.{}

\end{abstract}

\maketitle

\section{Introduction}

For the terminology and notation unexplained in the Introduction, see Section~2.
Recall that a Banach space $X$ is {\sl injective} if every linear bounded operator from any Banach space with
$X$ as a codomain extends to any superspace
of its domain. A~Banach space $X$ is {\sl automorphic} if any isomorphism between any two
of its subspaces of the same codimension extends to an automorphism of $X$.

Banach spaces of the form $C(K)$ for $K$ extremally disconnected (i.e.,
the Stone spaces of complete Boolean algebras, so for example
$\ell_\infty(\kappa)\equiv C(\beta\kappa)$) are known to be injective but it is not known if
there are injective Banach spaces which are not isomorphic to such $C(K)$s. 
The Banach spaces $c_0(\kappa)$, $\ell_2(\kappa)$ for infinite cardinals $\kappa$ are known to be automorphic
(\cite{auto-linros}, \cite{automorphic}) but it is not known if these are the only examples of automorphic Banach spaces
(cf. Questions \ref{questions1} (1)).
Several restricted versions of injective and automorphic properties are considered in the literature more
or less explicitly:

\begin{definition}[\cite{book}]\label{def-main}
Let $X$ and $Y$ be Banach spaces, let $\kappa$ be an infinite cardinal and
let $\B_\kappa$ denote the class of Banach spaces of density less than $\kappa$.
We say that
\begin{itemize}
\item $X$ is universally  separably injective [universally $\kappa$-injective] (universally $Y$-injective)  if
for every separable space $Z$ [for every $Z\in \mathcal B_\kappa$] (for every $Z$ isomorphic to $Y$) and every 
superspace $Z'$ of $Z$ 
every  linear bounded operator $T: Z\rightarrow X$ extends to an operator  $ \tilde T: Z'\rightarrow X$.
\item $X$ is separably extensible [$\kappa$-extensible] ($Y$-extensible) if for every
separable subspace $Z\subseteq X$ [for every subspace $Z\subseteq X$ in $\mathcal B_\kappa$]
(for every subspace $Z \subseteq X$ isomorphic to $Y$)
every linear bounded operator $T: Z\rightarrow X$ extends to an operator  $ \tilde T: X\rightarrow X$.
\item $X$ is separably automorphic [$\kappa$-automorphic] ($Y$-automorphic)
if for every isomorphic subspaces $Z, Z'\subseteq X$ which are separable [are in $\mathcal B_\kappa$]
(are isomorphic to $Y$)
of the same codimension (i.e.,  $dens(X/Z)=dens(X/Z')$)
 every isomorphism between $Z$ and $Z'$ extends to an automorphism of $X$.
\end{itemize}

\end{definition}
For example, Johnson and Zippin proved in \cite{john-zip} that $c_0(\kappa)$ is $Y$-extensible 
for any subspace $Y\subseteq c_0(\kappa)$ and any infinite cardinal $\kappa$, or Lindenstrauss and Rosenthal
proved in \cite{auto-linros} that $\ell_\infty$ is $2^\omega$-automorphic and it is not automorphic, while
in \cite{automorphic} Moreno and Plichko proved that $\kappa$-automorphic ($Y$-automorphic) 
space must be $\kappa$-extensible ($Y$-extensible) for any infinite cardinal $\kappa$
(and any Banach space $Y$). Also, clearly, 
universal \(\kappa\)-injectivity (universal \(Y\)-injectivity) 
implies \(\kappa\)-extensibility (\(Y\)-extensibility) so extensibility is the 
weakest property of all three. \color{black} A~modern account 
and many new results on such restricted injective and automorphic
properties focused on separable injectivity are presented in the monograph
\cite{book} by Avil\'es,  Cabello S\'anchez, Castillo, Gonz\'alez and  Moreno.

In this paper we focus on the above properties for the Banach space $\ell_\infty/c_0\equiv C(\beta\N\setminus \N)$, which
is well known to be rather sensitive to additional set-theoretic assumptions
(\cite{book, universal, christina, dow, drewnowski, cristobal, krupski, stevo}).
It had been an open problem whether $\ell_\infty/c_0$
is an injective Banach space. It was solved in the negative by Amir in \cite{amir}.
In fact, Rosenthal's proof of this result (Proposition  1.2 of \cite{rosenthal}) shows
that $\ell_\infty/c_0$
is not universally $c_0(2^\omega)$-injective.
On the other hand,
it is shown in \cite{book} that $\ell_\infty/c_0$ is universally separably injective and separably automorphic,
 so also separably extensible. 

Fundamental phenomena of uncountable combinatorics already discovered by Hausdorff and Luzin imply that the
 Boolean algebra $\wp(\N)/Fin$ corresponding to $\ell_\infty/c_0$
is not ${\mathit{Fincofin}}(\omega_1)$-injective, nor ${\mathit{Fincofin}}(\omega_1)$-extensible, nor 
${\mathit{Fincofin}}(\omega_1)$-automorphic.\footnote{The definitions of these notions for Boolean algebras 
are analogous to those from Definition \ref{def-main}.}
Motivated by that, the authors of \cite{book} proved the failure of isometric versions of injectivity
of $\ell_\infty/c_0$ already at the first uncountable level $\omega_1$ (Proposition 5.18 of \cite{book}). The same motivation
underlies the main open problem (Problem 12) discussed
in Section 6.4.3. of \cite{book} whether the injective properties of $\ell_\infty/c_0$ as in Definition \ref{def-main}
already fail at the first uncountable level $\omega_1$ possibly due to some similar, perhaps
yet undiscovered, combinatorial phenomena or they may consistently hold, for example,
up to arbitrarily big continuum. Discussing this problem in section 6.4.3 the authors of \cite{book}
write:
\vskip 6pt
{\it ``the problem of injectivity-like properties of $\ell_\infty/c_0$
is connected at a deep level
with its $c_0(\omega_1)$-automorphic character''}
\vskip 6pt
and they state an auxiliary open Problem 13 whether $\ell_\infty/c_0$ is 
$c_0(\omega_1)$-automorphic. The first main result of this paper
shows that this problem is undecidable by the usual axioms of {\sf ZFC}. This is
a consequence of the following two theorems:

\begin{theorem}\label{main-a} If $\mathfrak a=\omega_1$, then
$\ell_\infty/c_0$ is not $c_0(\omega_1)$-automorphic nor $c_0(\omega_1)$-extensible\footnote{Recall that
$\mathfrak a$ stands for the smallest infinite cardinality of a maximal family of infinite subsets of $\N$
whose pairwise intersections are finite (maximal almost disjoint family). It is known that
$\mathfrak a=\omega_1$ is consistent with arbitrarily big continuum (see e.g. \cite{douwen}, \cite{blass}).}.
\end{theorem}

\begin{proof}
 Let \(X = \overline{\mr{span}}\{1_{[A_{\xi}]} : \xi < \kappa\}\) 
for some almost disjoint family \(\{A_\xi : \xi \in \kappa\} \subseteq \wp(\N)\). 
We will justify that \(X\) has codimension \(2^\omega\) in \(C(\N^*)\). 
Choose any \(A \subseteq A_1\) such that \(A\) and \(A_1\setminus A\) 
are infinite, and denote \(Y = \{f \in C(\N^*) : \supp f \subseteq A\}\). 
The quotient from map from $C(\N^*)$ anto $C(\N^*)/X$
restricted to \(Y\) is an isomorphism and so \(\mr{dens}(C(\N^*)/X) \geq \mr{dens}(Y) = 2^{\omega}\).
To conclude, use Proposition \ref{embedding} and Theorem \ref{t-noextension}. \color{black}
\end{proof}

As the hypothesis $2^{\kappa}>2^\omega$ easily implies, via Rosenthal's argument, that
$\ell_\infty/c_0$ is not universally $c_0(\kappa)$-injective, it should be noted that
the hypothesis \(\mathfrak{a} = \omega_1\) of Theorem \ref{main-a} is consistent 
with $2^{\omega_1} =2^{\mathfrak{a}}=2^\omega$ (this holds e.g., in the Cohen model), so  Theorem \ref{main-a} gives also
a new information about the failure  of $c_0(\kappa)$-injectivity of $\ell_\infty/c_0$ as well.

\begin{theorem}\label{main-ma} Suppose that $\kappa$ is an infinite cardinal.
    Assume $\mathsf{MA}_{\hbox{$\sigma$-{\small{\rm linked}}}}(\kappa)$. 
    Then $\ell_\infty/c_0$ is $c_0(\kappa)$-automorphic and so $c_0(\kappa)$-extensible. 
 \end{theorem}
\begin{proof} Use Proposition \ref{ma}
\end{proof}

Although we were not able to fully solve Problem 12 of \cite{book},
the above quoted passage from \cite{book} turned out to be quite accurate, as 
using Theorem \ref{main-ma} and an old result of \cite{bfz} we were able to conclude
the consistency of a strong unknown injective property of $\ell_\infty/c_0$:

\begin{theorem}\label{main-inj} It is consistent with arbitrarily big continuum that for
every $\kappa<2^\omega$ every copy of $c_0(\kappa)$ in $\ell_\infty/c_0$
is included in a copy of $\ell_\infty(\kappa)$, and so it is consistent that $\ell_\infty/c_0$
is $c_0(\kappa)$-injective.
\end{theorem}

\begin{proof} By the main result of \cite{bfz} $\mathsf{MA}_{\hbox{$\sigma$-{\small{\rm linked}}}}$
with  arbitrarily big continuum  is consistent with the existence of a Boolean embedding
of any Boolean algebra of cardinality not exceeding $2^\omega$ into the algebra $\wp(\N)/Fin$,
in particular this applies to the algebras $\wp(\kappa)$ for every $\kappa<2^\omega$ as
$2^\kappa=2^\omega$ for such $\kappa$s under 
$\mathsf{MA}_{\hbox{$\sigma$-{\small{\rm centered}}}}(\kappa)$. Hence,
in this situation, there are isometric embedding
$T_\kappa:\ell_\infty(\kappa)\rightarrow\ell_\infty/c_0$ for such $\kappa$s. Consider
any copy $X\subseteq \ell_\infty/c_0$  of $c_0(\kappa)$. By Theorem \ref{main-ma}
there is an automorphism $S: \ell_\infty/c_0\rightarrow \ell_\infty/c_0$ such that
$S[X]=T_\kappa[c_0(\kappa)]$. So $S^{-1}[T_\kappa[\ell_\infty(\kappa)]]$ is the required
copy of $\ell_\infty(\kappa)$ containing $X$.

To prove $c_0(\kappa)$-injectivity,
let $R:c_0(\kappa)\rightarrow \ell_\infty/c_0$ be any bounded operator and let $Z$ be any superspace of
$c_0(\kappa)$.
Let $R': T_\kappa [c_0(\kappa)]\rightarrow \ell_\infty/c_0$    be such that $R=R'\circ T_\kappa|  c_0(\kappa)$.
$R'$ can be extended to $R'': T_\kappa[\ell_\infty(\kappa)]\rightarrow \ell_\infty/c_0$
 since $\ell_\infty/c_0$ is $c_0(\kappa)$-automorphic 
by Theorem \ref{main-ma} and so $c_0(\kappa)$-extensible by \cite{automorphic}. 
On the other hand, by the injectivity of $\ell_\infty(\kappa)$, there is an operator $U: Z\rightarrow T_\kappa[\ell_\infty(\kappa)]$
extending $T_\kappa| c_0(\kappa)$. So $R''\circ U$ is the desired extension of $R$.
\end{proof}

The above results should be seen in the Boolean context, where
analogous properties are false in {\sf ZFC} already for $\kappa=\omega_1$.
This is witnessed for example by the Luzin family (see e.g. \cite{hrusak}, \cite{douwen}),
which does not admit any separation into two uncountable parts.
Our results imply, for example, that it is consistent that if one considers a Luzin family $\{A_\xi: \xi<\omega_1\}$
and any other almost disjoint family $\{B_\xi: \xi<\omega_1\}$ which admits
a separation into two uncountable parts, then there
is an automorphism $S$ of $\ell_\infty/c_0$ such that $S(1_{[A_\xi]})=1_{[B_\xi]}$ for
every $\xi<\omega_1$, while there cannot be a Boolean automorphism of 
$\wp(\N)/Fin$ which sends $[A_\xi]$ to $[B_\xi]$.

A similar topic which makes sense in the Boolean context is the question whether
all or no nontrivial automorphisms of a fixed subalgebra of $\wp(\N)/Fin$  like ${\mathit{Fincofin}}(\omega_1)$
extend to an automorphism of $\wp(\N)/Fin$. This is investigated in \cite{auto-alan} and \cite{grzech},
however, the results for ${\mathit{Fincofin}}(\omega_1)$ are negative in {\sf ZFC} as well (Proposition 4.3 of \cite{auto-alan}).

The structure of the paper is as follows. First, in the next section we present notation and terminology
and discuss some preliminaries. In Section 3 we deal with isometric
embeddings of $\ell_\infty^c(\kappa)$ into $\ell_\infty/c_0$ for uncountable $\kappa$. We prove that
there is in {\sf ZFC} an almost disjoint family $\B=\{B_\xi: \xi<\omega_1\}$ such that
there is an isometric embedding $S:\ell_\infty^c(\omega_1)\rightarrow C(\N^*)$ satisfying
$$S(1_{\{\xi\}})=1_{[B_\xi]}$$
for every $\xi<\omega_1$ (Proposition \ref{embedding}). Moreover we obtain:

\begin{theorem}\label{b-equiv}
The following conditions are equivalent\footnote{Recall that the bounding number $\mathfrak b$ is the minimal
cardinality of a family of functions in $\N^\N$ which is unbounded in the order of eventual 
pointwise domination. It is known that $\omega<\mathfrak b\leq\mathfrak a\leq 2^\omega$ (see \cite{blass, douwen}).}
\begin{enumerate}
\item $\mathfrak b=\omega_1$,
\item There is an almost disjoint family $\A=\{A_\xi: \xi<\omega_1\}$ such that
there is no isometric embedding $S:\ell_\infty^c(\omega_1)\rightarrow C(\N^*)$ satisfying
$$S(1_{\{\xi\}})=1_{[A_\xi]}$$
for every $\xi<\omega_1$.
\end{enumerate}
\end{theorem}
\begin{proof}
Use propositions \ref{b=omega1} and \ref{b>omega1}.
\end{proof}
In the fourth section we consider obstructions to isomorphic
embeddings of $\ell_\infty^c(\kappa)$ into $\ell_\infty/c_0$ for uncountable $\kappa$.
We prove:

\begin{theorem}\label{t-noextension} There 
is an almost disjoint family $\A=\{A_\xi: \xi<\mathfrak a\}$ such that
there is no linear operator $S:\ell_\infty^c(\mathfrak a)\rightarrow C(\N^*)$ satisfying
$$S(1_{\{\xi\}})=1_{[A_\xi]}$$
for every $\xi<\mathfrak a$. In particular,
\(\ell_\infty/c_0\) is not universally \(c_0(\mathfrak{a})\)-injective.
\end{theorem}
\begin{proof} Use Proposition \ref{noextension}.
\end{proof}

The results of sections 3 and 4, in particular, give the proof of Theorem \ref{main-a}.
In the 5th section we prove results needed to conclude Theorem \ref{main-ma}.
The automorphism of $\ell_\infty/c_0$ constructed in the proof of Theorem \ref{main-ma}
is obtained from an automorphism of $\ell_\infty$ which preserves $c_0$, i.e., can be lifted to $\ell_\infty$.
Such automorphisms of $\ell_\infty$ are represented by appropriate infinite matrices (see Section 2.1.).
We use a restricted version of Martin's axiom to build such a matrix as a block diagonal matrix with finite blocks.
The construction of the finite blocks is based on a lemma of Bourgain and Tzafriri (see Section 2.3)
concerning extensions of isomorphisms from small subspaces of finite dimensional Banach spaces.

It should be noted that already the paper \cite{extensions} of Marciszewski and Plebanek
includes the use of Martin's axiom to construct a linear operator on $\ell_\infty$
which sends a given almost disjoint family of cardinality $\kappa<2^\omega$
to any other one of the same cardinality (which has been later exploited in \cite{sailing}). 
However, it seems that these operators
may not preserve $c_0$ and anyway the results of \cite{extensions}   do not apply to
arbitrary isomorphic copies of $c_0(\kappa)$. 

In particular, given any two Johnson-Lindenstrauss' spaces of the same density
our method of proving Theorem \ref{main-ma} gives
an explicit $\N\times \N$ matrix which sends one of them onto the other. 
In fact, if we face two almost disjoint families of cardinality $\kappa<2^\omega$ instead of
arbitrary copies of $c_0(\kappa)$ our method of constructing the matrix 
can be carried out without need of the lemma of Bourgain and Tzafriri.

The last 6th section contains some questions.

\section{Preliminaries}

\subsection{General terminology and notation}

If $A$ is a set, by $1_A$ we denote its characteristic function defined on an implicitly given domain.
If $f$ is a function, $f|A$ denotes the restriction of $f$ to $A$ and $\ran(f)$ denotes the range of $f$.
By a~simple function, we mean a function that assumes finitely many values.

Let $\wp(X)$ denote the powerset of a set $X$. We will sometimes consider $\wp(X)$ as a Boolean algebra. By ${\mathit{Fincofin}}(X)$ we mean the Boolean subalgebra
of all finite and cofinite subsets of $X$. $Fin$ stands for
the ideal of all finite subsets of $\N$.
The symbols $\vee, \wedge, -$ denote Boolean operations.
   
All Banach spaces that appear in this paper are considered over $\R$. The norms we consider are mainly supremum norms on adequate spaces. Another norm we use will be the quotient norm on \(\ell_\infty/c_0\). It should
be clear from the context of the Banach space which norm is used. The symbols $c_0(\Gamma), \ell_\infty(\Gamma), \ell_1(\Gamma)$ have the
standard meaning for any set $\Gamma$ and $\Gamma$ is omitted
if it is $\N$. By \(\supp f\) we mean the support of function \(f\) and \(c_{00}\) stands for the linear subspace of \(c_0\) of all finitely supported sequences. The subspace of $\ell_\infty(\kappa)$ consisting of countably
supported elements is denoted as $\ell_\infty^c(\kappa).$ We sometimes identify Banach spaces
$c_0(A), \ell_\infty(A), \ell_\infty^c(A), \ell_1(A)$ with the subspaces of $c_0(\Gamma), \ell_\infty(\Gamma), \ell_\infty^c(\Gamma), \ell_1(\Gamma)$ consisting 
of all elements whose supports are included in $A \subseteq \Gamma$, and elements \(f\) defined on a subset $A \subseteq \Gamma$ with \(f \cup \{0\}^{\N \setminus A}\).
To keep the notation consistent with the literature, in the case of a finite set $A$ the symbol $\ell_\infty^A$ 
stands for the finite dimensional space $\R^A$ with the maximum norm.

By an operator on a Banach space we always mean a bounded linear operator. We say that an operator $T$ is bounded below by $r\in \R_{>0}$ if
$r\|x\|\leq \|T(x)\|$ for each $x$ in the domain of $T$.
The algebra of operators on a Banach space $X$ is denoted by $L(X)$. The identity on $X$ is denoted by $I_X$. The symbol
$\equiv$ denotes the relation of linear isometry between Banach spaces. For isomorphic Banach spaces $X, Y$ we will also consider their Banach-Mazur distance $d(X, Y)$, that
is
$$\inf\{\|T\|\|T^{-1}\|: \hbox{$T$ is an isomorphism from $X$ onto $Y$}\}.$$
For $I\subseteq\N$ we will consider
$I \times I$ matrices, which are formally functions from $I\times I$ into $\R$ 
and may represent certain bounded linear operators on \(\ell_{\infty}(I)\) (and/or \(c_0(I)\)). For this, the rows of the matrix must be in $\ell_1$ and the $\ell_1$-norms of the rows
need to be bounded (converge to $0$) (cf. Section 2.1. of \cite{cristobal}). Norm of a matrix will always be the operator norm induced by the supremum norms in the domain
and the codomain. It should be clear that the operator norm of an operator on $c_0(I)$ or on 
$\ell_\infty(I)$ given by a matrix $M$ is the supremum of $\ell_1$-norms of the rows of $M$. We use the usual conventions concerning multiplication of matrices. An $I\times I$ matrix $M$ is called block diagonal if there is a partition of $I$ into finite sets $\{J_i: i\in A\}$
such that $M(j, j')=0$ whenever $j, j'$ belong to different parts of the partition. The $J_i\times J_i$-matrices
obtained from $M$ are called blocks of $M$. $M_n(\Q)$ denotes
$n\times n$ matrices with rational entries.
 
We write $A \subseteq^*{A'}$ if 
$A \setminus A'$ is finite and $A=^*A'$ if $A\subseteq^* A'$ and $A'\subseteq^* A$. 
We denote by $[A]$ the class of a set $A\subseteq \N$ 
with respect to relation $=^*$. Similarly, for two functions
$f, g \in \ell_{\infty}$ we write $f=^* g$ if $f=g + h$ for some $h \in c_0$ 
and we denote by $[f]$ the equivalence class of a function $f$ with respect to this relation.
We consider the quotient structures, the Boolean algebra $\wp(\N)/Fin$ and the Banach space $\ell_\infty/c_0$ (with the usual quotient norm).
Recall that, by Stone duality, we canonically identify any Boolean algebra \(\A\) with the Boolean algebra of all clopen subsets of its Stone space \(K_{\A}\). In particular, $\wp(\N)/Fin$ is identified with the Boolean algebra of
all clopen subsets of $\N^*$ and there is an isometry between $\ell_\infty/c_0$ and the Banach space $C(\N^*)$ of
all continuous functions on $\N^*=\beta\N\setminus \N$ which sends $[1_A]$ to
$1_{[A]}$ for $A\subseteq\N$.

A family  $\A$ of subsets of $\N$ is called an almost disjoint family
if for any $A, A' \in \A$, if $A\ne A'$ then $A \cap A'=^*\emptyset$.
Whenever we consider almost disjoint families, we always assume
that they are infinite and consist of infinite subsets of $\N$.
If \(\A, \B\) are two families of subsets of \(\N\), we say that $\A$ is separated from $\B$
if there is $C\subseteq \N$ such that $A\subseteq^* C$ and $B\cap C=^*\emptyset$
for all $A\in \A$ and $B\in \B$. Such a set $C$ is called a separation of $\A$ and $\B$.
Note that under the identification mentioned above, if $\A$ is an almost disjoint family, then $\{[A]: A\in \A\}$
is a pairwise disjoint family of clopen subsets of $\N^*$ and so the closure of the span
of $\{1_{[A]}: A\in \A\}$ is an isometric copy of $c_0(|\A|)$ in \(C(\N^*)\). 

A partial order $\PP$ is said to be $\sigma$-linked if $\PP=\bigcup\PP_n$, where
for each $n\in \N$ the set $\PP_n$ is linked, i.e., any two elements of $\PP_n$ are compatible in $\PP$.
$\mathsf{MA}_{\hbox{$\sigma$-{\small{\rm linked}}}}(\kappa)$ stands for a restricted version
of Martin's axiom which states that given a $\sigma$-linked partial order $\PP$ and
a family $\mathcal D$ of cardinality $\kappa$ of dense subsets of $\PP$ there is a filter $\mathbb G\subseteq\PP$
such that $\mathbb G\cap \mathbb D\not=\emptyset$ for every $\mathbb D\in \mathcal D$.
For more on the terminology concerning Martin's axiom, see \cite{kunen} or \cite{weiss}.

For further terminology and notation, the reader may consult books \cite{A&K}, \cite{book} and \cite{kunen}.

\subsection{The Banach space $\ell_\infty/c_0$ }      
         
\begin{definition}\label{def-pi} $\pi: \ell_\infty\rightarrow \ell_\infty/c_0$ is the quotient map.
\end{definition}

\begin{definition}\label{sim}
Let \(a\) be a finite set, \(\{ x_{\xi} : \xi \in a\} \subseteq X\), \(\{ y_{\xi}: \xi \in a\} \subseteq Y\) be two indexed sets of vectors in Banach spaces \(X, Y\) respectively and let \(0 < p, q < \infty\) be positive real numbers. We write \(\langle x_{\xi} : \xi\in a\rangle \overset{\scriptscriptstyle{p, q}}\sim \langle y_{\xi} : \xi \in a\rangle\) if for any sequence \(t =(t_\xi)_{\xi\in a} \in \ell_\infty^{a}\) we have
\[
p^{-1}\left\|\sum_{\xi \in a} t_{\xi}y_{\xi}\right\|_{Y} \leq \left\|\sum_{\xi \in a} t_{\xi}x_{\xi}\right\|_{X} \leq q\left\|\sum_{\xi \in a} t_{\xi}y_{\xi}\right\|_{Y}.
\]
\end{definition}

We note a few simple but useful observations.

\begin{lemma}\label{equivsim}
Denote \(X' = \mr{span}\{x_\xi : \xi \in {a}\}\) and \(Y' = \mr{span}\{y_\xi : \xi \in {a}\}\). Then the following are equivalent:
\begin{enumerate}
    \item \(\langle x_{\xi} : \xi \in a\rangle \simi{p, q} \langle y_{\xi} : \xi \in a \rangle\),
    \item there is an isomorphism \(V:X' \to Y'\) such that \(V(x_{\xi}) = y_\xi\) which is bounded from above by \(p\) and from below by \(q^{-1}\),
    \item there is an isomorphism \(V:X' \to Y'\) such that \(V(x_{\xi}) = y_\xi\) with \(\norm{V} \leq p\) and \(\norm{V^{-1}} \leq q\).
\end{enumerate}
\end{lemma}

\begin{lemma}\label{sim-prop}
    The following properties hold:
    \begin{enumerate}
        \item \(\langle x_{\xi} : \xi\in a\rangle \simi{p, q} \langle y_{\xi} : \xi \in a\rangle\) if and only if \(\langle y_{\xi} : \xi \in a\rangle \simi{q, p} \langle x_{\xi} : \xi\in a\rangle\),
        \item if \(\langle x_{\xi} : \xi\in a\rangle \simi{p, q} \langle y_{\xi} : \xi \in a\rangle \) and \( \langle y_{\xi} : \xi \in a\rangle \simi{p', q'} \langle z_{\xi} : \xi\in a\rangle\), then \(\langle x_{\xi} : \xi\in a\rangle \simi{pp', qq'} \langle z_{\xi} : \xi\in a\rangle,\)
        \item if \(\langle x_{\xi} : \xi\in a\rangle \simi{p, q} \langle y_{\xi} : \xi \in a\rangle\) and \(a' \subseteq a\), then \(\langle x_{\xi} : \xi\in a'\rangle \simi{p, q} \langle y_{\xi} : \xi \in a'\rangle\).
    \end{enumerate}
\end{lemma}

\begin{lemma}\label{lifting} Suppose that $a$ is a finite set and $\{ y_\xi : \xi \in a\} \subseteq \ell_\infty$ is such
that \(\{\pi(y_\xi) : \xi \in a\}\) is a linearly independent set.  Let $u>1$.
There exists $N\in \N$ such that \(\ang{y_\xi|[n, \infty) : \xi \in a} \simi{1, u} \ang{\pi(y_\xi) : \xi \in a}\) for all $n>N$.
\end{lemma}

\begin{proof}
Define \(f_n(t) = \sum_{\xi \in a} t_\xi y_\xi|{[n, \infty)}\) for \(n \in \N\) and \(t \in \ell_\infty^a\). We always have \(\norm{f_n(t)} \geq \norm{\pi(f_0(t))}\). For \(n \in \N\) let
\[
E_{n} = \{t \in S_{\ell_\infty^a} : u^{-1}\cdot \|f_n(t)\| < \|\pi(f_0(t))\|\}.
\]
By the continuity of \(\pi\), \(f_n\) and the corresponding
norms, the sets \(E_n\) are open. For every \(t \in S_{\ell_\infty^a}\) we 
have \(\|f_n(t)\| \searrow \|\pi(f_0(t))\|\), by the very 
definition of \(f_n\)s and the norms, and \(\norm{\pi(f_0(a))} > 0\), by 
the fact that \(\{\pi(y_\xi) : \xi \in a\}\) is linearly independent.
Consequently, \(E_n \subseteq E_m\) 
whenever \(n < m\) and for every \(t \in S_{\ell_\infty^a}\) there is \(n \in \N\) 
such that \(t \in E_n\). Since \(S_{\ell_\infty^a}\) is compact, there is \(N \in \N\) such that \(S_{\ell_\infty^{a}} = \bigcup_{n \leq N} E_n = E_N\). Then \(u^{-1}\cdot \|f_n(t)\| < \|\pi(f_0(t))\|\) whenever \(n \geq N\) for all \(t \in S_{\ell_\infty^{a}}\), so also, by linearity, for all \(t \in \ell_\infty^{a}\), which gives us the thesis.
\end{proof}

\begin{lemma}\label{bounded-restriction} Suppose that \(a\) is a finite set and $\{y_\xi : \xi \in a\} \subseteq \ell_\infty(A)$ for some subset \(A \subseteq \N^*\). Let $u>1$.
Then there is $N\in \N$ such that \(\ang{y_\xi : \xi \in a} \simi{1, u} \ang{y_\xi|(A\cap n) : \xi \in a}\) for every $n>N$.
\end{lemma}

\begin{proof}
Denote \(Y = \mr{span}\{y_\xi : \xi \in a\}\), \(Y_n = \mr{span}\{y_\xi|(A\cap n) : \xi \in a\}\) and define \(P_n: Y \to Y_n\) as \(P(y) = y|(A \cap n)\). Clearly, \(P_n\)s are bounded from above by 1. Let
\[
E_{n} = \{y \in S_Y: u^{-1}\cdot \|y\| < \|y|{(A \cap n)}\|\}.
\]
For every \(y \in S_Y\) we have \(\|y|{(A \cap n)}\| \nearrow \|y\| > 0\), 
so \(y \in E_n\) for some \(n \in \N\). Since \(S_Y\) is compact and 
\(E_n\)s are open, there is \(N \in \N\) such that \(S_Y = \bigcup_{n \leq N} E_n = E_{N}\). 
Then for \(n \geq N\) the operator \(P_n\) is bounded below by \(u^{-1}\) on \(S_Y\) so, by linearity, on the whole space \(Y\) as well. By Lemma \ref{equivsim}, this finishes the proof.
\end{proof}

\subsection{Around a theorem of Bourgain and Tzafriri}

\begin{lemma}[Theorem 5 of \cite{bourgain}]\label{bourgain}
For every $K\geq1$, there exist constants $c =c(K) > 0$ and $C=C(K)<\infty$
such that, whenever 
$$\ell^n_p=Y_1\oplus Z_1= Y_2\oplus Z_2$$
for some integer $n$ and some $1\leq p\leq\infty$, and
\begin{enumerate}
\item[(i)] $dim(Y_1) =dim(Y_2) =h$ and $d(Y_1, Y_2)\leq K$,
\item[(ii)]\footnote{$P_{Y_i}$, $P_{Z_i}$ are projections determined
by the decompositions $\ell^n_p=Y_i\oplus Z_i$ for $i=1,2$.} $\|P_{Y_i}\|, \|P_{Z_i}\|\leq K$ for $i=1,2$,
\item[(iii)] $h\leq c\sqrt{n}$,
\end{enumerate}
then
$$d(Z_1, Z_2)\leq C.$$
\end{lemma}

\begin{lemma}\label{bourgain2} For every $\rho > 1$ there are $0 < c_1(\rho), c_2(\rho)<\infty$ such that
whenever $Y_1, Y_2$ are subspaces of $\ell^n_\infty$ for some $n\in \N$ and \(T: Y_1 \to Y_2\) is an isomorphism which satisfy
\begin{enumerate}
\item $dim(Y_1)=dim(Y_2)=h\leq c_1(\rho)\sqrt{n}$,
 \item $d(Y_i, \ell_{\infty}^h) \leq \rho$ for $i=1,2$,
\item $\|T\|, \|T^{-1}\|\leq\rho$, 
\end{enumerate}
then there is an automorphism $\tilde T$ of $\ell^n_\infty$ extending \(T\) such that $\|\tilde T\|, \|\tilde T^{-1}\|<c_2(\rho)$.

Moreover, if $Y_1, Y_2$ are generated by vectors $\{v_k: k<h\}$ and $\{w_k: k<h\}$,
respectively, all whose coordinates are rational, and
$T(v_k)=w_k$ for all $k<h$, then we can choose $\tilde T$ as above whose matrix has all entries rational.
\end{lemma}

\begin{proof}
Let \(S_i: Y_i \to \ell_{\infty}^h\) for \(i = 1, 2\) be isomorphisms such that
 \(\|S_i \| \cdot \|S_i^{-1}\| \leq \rho\). 
 Let $\phi^i_j=S_i^*(e_j^*)$ for $1\leq j\leq h$ and $i=1,2$. So
$\|\phi^i_j\|\leq \|S_i\|$. Using the Hahn-Banach theorem, extend each $\phi^i_j$ to $\psi^i_j\in {(\ell_\infty^n)}^*$
preserving the norm. Define $P_{Y_i}: \ell_\infty^n\rightarrow Y_i$ by
$$P_{Y_i}(f)=S_i^{-1}(\psi^i_1(f), \dots, \psi^i_{h}(f))$$
for $f\in \ell^n_\infty$, $i=1,2$. It is clear that $P_{Y_i}$ is a projection
from $\ell^n_\infty$ onto $Y_i$ of the norm not exceeding $\rho$, so $P_{Z_i} = I-P_{Y_i}$  is a projection
onto $\mathrm{ker}(P_{Y_i})=Z_i$ of the norm not exceeding $1+\rho$. By part (3) of the hypothesis
$d(Y_1, Y_2)\leq \rho^2$. 

So we are in a position to apply Lemma \ref{bourgain} for $K=\rho^2+1$. The conclusion is
that there are $c(\rho^2+1), C(\rho^2+1)$ such that if it happens that $h\leq c(\rho^2+1)\sqrt{n}$,
then $d(Z_1, Z_2)\leq C(\rho^2+1)$. So there is an isomorphism
\(Q: Z_1 \to Z_2\) such that \(\norm{Q} \cdot \norm{Q^{-1}} \leq C(\rho^2 +1)\). 
Let $R=\sqrt{\|Q^{-1}\|/\|Q\|}Q$. Then \(R^{-1} = 
\sqrt{\|Q\|/\|Q^{-1}\|}Q^{-1}\) and $\|R\|=\|R^{-1}\|=\sqrt{\|Q\|\cdot\|Q^{-1}\|}
\leq \sqrt{C(\rho^2 + 1)}$.
Now define $\tilde T: \ell^n_\infty\rightarrow \ell^n_\infty$ by
$$\tilde T=TP_{Y_1}+ RP_{Z_1}.$$
It is clear that $\tilde T|Y_1=T$. Also
$$\tilde T(T^{-1}P_{Y_2}+ R^{-1}P_{Z_2})=(TP_{Y_1}+ RP_{Z_1})(T^{-1}P_{Y_2}+ R^{-1}P_{Z_2})=$$
$$=TT^{-1}P_{Y_2}+0+0+ RR^{-1}P_{Z_2}=P_{Y_2}+P_{Z_2}=I .$$
We obtain  $(T^{-1}P_{Y_2}+ R^{-1}P_{Z_2})\tilde T=I$ in an analogous way. So
$${\tilde T}^{-1}=T^{-1}P_{Y_2}+ R^{-1}P_{Z_2}.$$
As a result, we have
$$\|\tilde T\|, \|{\tilde T}^{-1}\|\leq \rho^2 + \sqrt{C(\rho^2+1)}\cdot (\rho+1)$$ 
So
{}$c_2(\rho)=\rho^2 + \sqrt{C(\rho^2+1)}\cdot (\rho+1)+1$ {} and 
{}$c_1(\rho)=c(\rho^2+1)${} work.

To prove the last part of the lemma, choose vectors \(\{v_k : h \leq k < n\}\)
with rational coordinates such that \(\mathcal{B} = \{v_k : k<n\}\) 
is a basis for \(\ell_{\infty}^n\). 
The set
\[
\mathcal{V} = \{U \in L(\ell_{\infty}^n) : U \text{ invertible}; \norm{U}, \norm{U^{-1}} < c_2(\rho)\}
\] 
is open in the norm topology on \(L(\ell_\infty^n)\) and the operator \(U \xmapsto{f} \mathcal{M}_{\mathcal{B}}^{st}(U)\),
which assigns to every \(U \in L(\ell_{\infty}^n)\) the matrix of \(U\)
with respect to the basis \(\mathcal{B}\) and the standard basis, is a 
continuous linear isomorphism, where the space of matrices is also considered with the norm topology. As a result, \(f[\mathcal{V}]\) is also open. 
Moreover, for every \(k < h\), \(Uv_k = w_k\) if and only if \(\mathcal{M}_{\B}^{st}(U)(i,k) =
w_k(i)\) for all \(i < n\). Choose any \(\tilde T\) from the thesis of the 
first part of the lemma. The matrix \(\mathcal{M}_{\mathcal{B}}^{st}(\tilde T)\) 
has rational entries in the first \(h\) columns and \(\mathcal{M}_{\mathcal{B}}^{st}(\tilde T) 
\in f[\mathcal{V}]\), so we can find \(A \in f[\mathcal{V}]\)
with all entries rational and \(w_k(i)=\mathcal{M}_{\mathcal{B}}^{st}(\tilde T)(i, k) 
= A(i, k)\) for \(k < h, i < n\).
By the assumption that \(\{v_{k} : k < n\}\) are
vectors with rational coordinates, \(\mathcal{M}_{\mc{B}}^{st}(Id) \in M_n(\Q)\). 
Since in the standard algorithm for inverting a matrix the entries of \(\mathcal{M}_{st}^{\mc{B}}(Id) = \mathcal{M}_{\mc{B}}^{st}(Id)^{-1}\) 
are obtained from the entries of \(\mathcal{M}_{\mc{B}}^{st}(Id)\) by a number of multiplications, 
additions and divisions of these entries (for instance, using Gaussian elimination), 
we know that also \(\mc{M}_{st}^{\mc{B}}(Id) \in M_n(\Q)\). If \(\tilde U = f^{-1}(A)\), then
\[\mathcal{M}_{{st}}^{st}(\tilde U) = \mathcal{M}_{st}^{\mathcal{B}}(Id)
\cdot \mathcal{M}_{\mathcal{B}}^{st}(\tilde U) \in M_n(\Q).\]
\end{proof}

\section{On isometric embeddings of $\ell_\infty^c(\omega_1)$ into $C(\N^*)$}

\begin{proposition}\label{embedding}
    There exists an almost disjoint family \(\{A_{\xi} : {\xi < \omega_1}\} 
    \subseteq \mathcal{P}(\N)\) and an isometric embedding
     \(S: \ell_{\infty}^{c}(\omega_1) \to C(\N^*)\) such that $S(1_{\{\xi\}})=1_{[A_\xi]}$.
\end{proposition}
  
\begin{definition}\label{def-coherent}
Let \(\{s_{\alpha} : \alpha < \omega_1\}\) be a family of functions. 
We call it \emph{a coherent family} if for each \(\alpha < \omega_1\):
\begin{enumerate}
    \item[(i)] \(s_{\alpha}: \omega \cdot \alpha \to \omega\) is injective,
    \item[(ii)] \(s_{\alpha}|{\omega \cdot \beta} =^* s_{\beta}\) for every \(\beta < \alpha\).
\end{enumerate}
\end{definition}
A construction of such a family is given for instance
in \cite[Chapter II Theorem 5.9]{kunen} (actually, our family is a subfamily
of the family constructed there, consisting only of functions enumerated with limit ordinals). 
The following lemma ensures that having such functions 
we are able to find an almost disjoint family with desired properties. 

\begin{lemma}\label{coherent-embedding}
Suppose that \(\{s_{\alpha} : \alpha < \omega_1\}\) is a coherent family.
Let
\[
A_{\xi} \deff s_{\xi + 1}\big[[\omega \cdot \xi, \omega \cdot \xi +\omega)\big].
\] 
Then \(\{A_{\xi} : \xi < \omega_1\}\) is an
almost disjoint family and there is an isometric embedding
\(S: \ell_{\infty}^{c}(\omega_1) \to C(\N^*)\) such that $S(1_{\{\xi\}})= 1_{[A_{\xi}]}$.
\end{lemma}

\begin{proof}
For \(\xi < \omega_1\) let us denote \(\xi^+ = \xi + 1\) and \(I_{\xi} = [\omega \cdot \xi, \omega \cdot \xi^+)\). 
Under this notation, \(A_\xi = s_{\xi^+}[I_\xi]\).  Since \(\{s_{\alpha} : \alpha < \omega_1\}\) is a coherent family, for \(\xi > \eta\) we have
\[
A_{\xi} \cap A_{\eta} = s_{\xi^+}[I_\xi] \cap s_{\eta^+} [I_\eta] =^* s_{\xi^+}[I_\xi] \cap s_{\xi^+}[I_\eta] = s_{\xi^+}[I_\xi \cap I_\eta] = s_{\xi^+}[\emptyset] = \emptyset,
\]
so the family \(\{A_\xi : \xi < \omega_1\}\) is indeed almost disjoint. Let $\A$ be the Boolean algebra of all countable and cocountable subsets of $\omega_1$ and let
$$U_X=\bigcup_{\xi\in X}I_\xi,$$
for $X\subseteq\omega_1$. Now note that to prove the lemma it is enough to construct a Boolean monomorphism $h:\A\rightarrow\wp(\N)/Fin$
such that $h(\{\xi\})=[A_\xi]$ for every $\xi<\omega_1$. This is because such a monomorphism
induces an isometry $T_h: C(K)\rightarrow C(\N^*)$
satisfying $T_h(1_{X})=1_{h(X)}$ for $X\subseteq\omega_1$ countable or cocountable, where $K$ is the Stone space of $\A$. The isometry is determined on the dense subspace of continuous simple functions on $K$. Note that, under the standard identification of elements of \(\A\) with clopen subsets of \(K\), \(\ell_{\infty}^c(\omega_1)\) is the closed subspace of \(C(K)\) generated by the characteristic functions \(1_X\) for \(X\) countable. Since 
$$
T_h(1_{\{\xi\}})=1_{h(\{\xi\})}=1_{[A_\xi]}
$$
for every $\xi<\omega_1$, $S=T_h|\ell_\infty^c(\omega_1)$ would work for the second part of the lemma.

From now on we focus on constructing $h$ as above. For a countable $X\subseteq\omega_1$ define $\alpha(X)=\min\{\alpha<\omega_1: U_X\subseteq \omega\cdot\alpha\}$.
Define $h: \A\rightarrow \wp(\N)/Fin$ by
\[
h(X)=
\begin{cases}
[s_{\alpha(X)}[U_X]] & \hbox{\rm if}\ X\ \hbox{\rm is  countable}, \\
[\N\setminus s_{\alpha(\omega_1\setminus X)}[U_{\omega_1\setminus X}]] & \hbox{\rm if}\ X\ \hbox{\rm is  cocountable}.
\end{cases}
\]
The definition clearly implies that \(h( - A) =  - h(A)\) for any \(A \in \A\). For countable $X, Y\subseteq\omega_1$ and a Boolean operation $\diamond \in \{\vee, \wedge, -\}$ we have
$$
s_{\alpha(X)}[U_X]\diamond s_{\alpha(Y)}[U_Y]=^*s_{\alpha(X \cup Y)}[U_X]\diamond s_{\alpha(X \cup Y)}[U_Y] =$$
$$ =s_{\alpha(X \cup Y)}[U_X\diamond U_Y]=s_{\alpha(X \cup Y)}[U_{X\diamond Y}]=^*s_{\alpha(X\diamond Y)}[U_{X\diamond Y}]$$
since $s_\alpha$s are injective and coherent by Definition \ref{def-coherent}. So $h(A\wedge B)=h(A)\wedge h(B)$, $h(A\vee B)=h(A)\vee h(B)$
and $h(A\setminus B)=h(A)-h(B)$ for countable $A, B\in \A$. From that, using the operation of complement, we can conclude that the same holds for arbitrary \(A, B \in \A\) and that $h$ defined above is a Boolean homomorphism. Moreover,
if $X\not=\emptyset$, 
then $U_{X}$ is infinite, so $h(X) = [s_{\alpha(X)}[U_{X}]] \ne \emptyset$. Thus, $h$ is a monomorphism.
\end{proof}

\begin{proposition}\label{iso-separation} 
Let \(\{A_{\xi} : \xi < \kappa\}\) be an almost disjoint family. 
If there exists an isometric embedding
\(S: \ell_{\infty}^{c}(\kappa) \to C(\N^*)\) 
such that $S(1_{\{\xi\}}) =  1_{[A_\xi]}$, 
then for every $F \in [\kappa]^{\leq \omega}$ there is \(V_{F} \subseteq \N\) 
such that \(A_{\xi} \subseteq^* V_{F}\) for 
\(\xi \in F\) and \(A_{\xi} \cap V_{F} =^* \emptyset\) for \(\xi \notin F\). 
\end{proposition}

\begin{proof}
Fix \(F \in [\kappa]^{\leq \omega}\).
Suppose that \(S( 1_{F})(x) < 1\) for some \(x \in [A_{\xi}], \, \xi \in F\). Then
\[
1 \geq S(2\cdot  1_{\{\xi\}} -  1_{F})(x) = 2\cdot  1_{[A_{\xi}]}(x) - S( 1_{F})(x) > 2-1 = 1,
\]
a contradiction. Similarly, if \(S (1_{F})(x) > 0\) for some \(x \in [A_{\xi}], \, \xi \notin F\), then
\[
1 \geq S( 1_{\{\xi\}} +  1_{F})(x) = 1_{[A_{\xi}]}(x) + S (1_{F})(x) > 1,
\]
which again gives a contradiction. We conclude that 
\([A_{\xi}] \subseteq \{x \in \N^* : S( 1_{F})(x) = 1\}\) 
for \(\xi \in F\) and
\([A_{\xi}] \cap \{x \in \N^* : S( 1_{F})(x) > 0\}= \emptyset\) 
for \(\xi \notin F\). 
By the compactness of $S(1_{F})^{-1}[\{1\}]$, 
there is a finite family \(\{U_i\}_{i=1}^n \subseteq \mathcal{P}(\N)\)
such that 
$$\{x \in \N^* : S( 1_{F})(x) = 1\} 
\subseteq \bigcup_{i=1}^n [U_i] 
\subseteq \{x \in \N^* : S(1_{F})(x) > 0\}.
$$
The set \(V_{F} = \bigcup_{i=1}^n U_i\) has the desired property.
\end{proof}
    
\begin{proposition}\label{b=omega1}
There is an almost disjoint family $\A=\{A_\xi: \xi<\mathfrak{b}\}$ such that
there is no isometric embedding  $S:\ell_\infty^{c}(\mathfrak{b})\rightarrow C(\N^*)$  satisfying
$$S(1_{\{\xi\}})=1_{[A_\xi]}$$
for every $\xi<\mathfrak{b}$.
\end{proposition}
\begin{proof}
By Theorem 3.3 from \cite{douwen} (case ${\mathfrak b}_5$) there exists an almost 
disjoint family  \(\mathcal{A} = \mathcal{B} \cup \mathcal{C} \subseteq [\omega]^{\omega}\) 
such that \(|\mathcal{C}| = \omega\), \(|\mathcal{B} | = \mathfrak{b}\), 
and families \(\mathcal{B} \) and \(\mathcal{C}\) cannot be separated. 
Then by Proposition \ref{iso-separation} for any enumeration \(\{A_{\xi} : \xi < \mathfrak{b}\}\) of 
\(\mathcal{A}\)  there is no isometric embedding 
$S:\ell_\infty^{c}(\mathfrak{b})\rightarrow C(\N^*)$ satisfying
$$S(1_{\{\xi\}})=1_{[A_\xi]}$$
for every $\xi<\mathfrak{b}$.
\end{proof}

\begin{lemma}\label{nice-ext}
Suppose that $A, B$ are countably infinite sets, $C\subseteq \N$ is infinite, and
\begin{enumerate}
\item $A\subseteq B$, $B\setminus A$ is infinite,
\item $f: A\rightarrow C$ is an injection,
\item $g: B\rightarrow \N$ is an injection, 
\item $\ran(g)\subseteq^*C$, 
\item $C\setminus \ran(g)$ is infinite
\item $f=^*g|A$
\item $F\subseteq C$ is finite.
\end{enumerate}
Then there is an injection $h:B\rightarrow C$ such that $f\subseteq h$, $h=^*g$ and $F\subseteq \ran(h)$.
\end{lemma}
\begin{proof}
Let $D_1=\{x\in B\setminus A: g(x)\in \N\setminus C\}$, so by (2)  $\ran(f\cup(g|E))\subseteq C$ 
for any $E\subseteq B\setminus A$ disjoint from $D_1$. By (3) and (4) $D_1$ is finite.
 
Let $D_2=\{x\in B\setminus A: \exists y\in A \  g(x)=f(y)\}$, so by (2) and (3)  $f\cup(g|E)$ is an injection
for any  $E\subseteq B\setminus A$ disjoint from $D_2$. By (6) $D_2$ is finite.
 
Let $D_3\supseteq D_1\cup D_2$ be any subset of $B\setminus A$ of finite cardinality not smaller than
$|F|$. It exists by (1). And let $E=(B\setminus A)\setminus D_3$. Finally let $h'=f\cup (g|E)$.
 
By the above remarks $h': A\cup E\rightarrow C$ is injective and $h'=^*(g|(A\cup E))$ by (6).
Now define $h'':D_3\rightarrow C\setminus\ran(h')$ by putting any distinct values on
distinct elements of $D_3$ as long as $\ran(h'')$ includes $F\setminus \ran(h')$. This can be done as $|D_3|\geq |F|$
and since $C\setminus \ran(h')$ is infinite by (5) and (6).
Now $h=h'\cup h''$ works.
\end{proof}

\begin{lemma}\label{iso-chain}
Let \(\{A_{\xi} : \xi < \omega_1\}\) be an almost disjoint family
and suppose that there is an almost increasing sequence \(\{V_{\alpha}\}_{\alpha < \omega_1} \subseteq \mathcal{P}(\N)\) such that \(A_{\xi}
\subseteq^* V_{\alpha}\) for \(\xi < \alpha\) and
\(A_{\xi} \cap V_{\alpha} =^* \emptyset\) for \(\omega_1>\xi \geq \alpha\). 
Then there is a coherent family \(\{s_{\alpha} : \alpha < \omega_1\}\) such that 
$$A_{\xi} =^* s_{\xi + 1}\big[[\omega \cdot \xi, \omega \cdot \xi +\omega)\big].$$
Consequently, by Lemma \ref{coherent-embedding}, there is an isometric embedding 
\(S: \ell_{\infty}^{c}(\omega_1) \to C(\N^*)\) such that $S(1_{\{\xi\}})=1_{[A_\xi]}$.
\end{lemma}

\begin{proof} Like in the proof of Lemma \ref{embedding}, we denote 
\(\xi^+ = \xi + 1\) and \(I_{\xi} = [\omega \cdot \xi, \omega \cdot \xi^+)\). 
For any $\alpha = \beta + k$ where \(\beta < \omega_1\) is a limit ordinal and \(k \in \omega\) consider  
$$
W_{\alpha} = V_{\beta} \cup \bigcup_{i<k}A_{\beta+i}.
\leqno \text{(i)}
$$
Note that the assumptions 
of the lemma concerning $\{V_{\alpha}\}_{\alpha < \omega_1}$ imply
the same properties for $\{W_{\alpha}\}_{\alpha < \omega_1}$, i.e., it is 
an almost increasing sequence such that \(W_{\alpha}\) is a separation of \(\{A_\xi : \xi < \alpha\}\) and \(\{A_{\xi}  :\alpha \leq \xi < \omega_1\}\) for every \(\alpha< \omega_1\). 
It follows that it is enough to find \(s_{\alpha} : \omega \cdot \alpha \to \omega\) which satisfy for any  $\alpha<\omega_1$:
\begin{enumerate}
\item[(ii)] $s_\alpha$ is an injection,
\item[(iii)] $s_{\alpha}|{\omega \cdot \beta} =^* s_{\beta}$ for every \(\beta < \alpha\),
\item[(iv)] $\ran(s_{\alpha})=W_\alpha$.
\end{enumerate}
Then by (i) - (iv) and the separation properties of \(\{W_\alpha\}_{\alpha< \omega_1}\)
$$
s_{\xi^+}[I_\xi]= \ran(s_{\xi^+}) \setminus s_{\xi^+}[\omega \cdot \xi] =^*
\ran(s_{\xi^+})\setminus\ran(s_{\xi})= W_{\xi^+} \setminus W_{\xi} =^*A_\xi.
$$
To construct $s_{\xi}$ we proceed in the standard way by induction on $\xi<\omega_1$. 
In the successor case, define $s_{\xi^+}$ as an extension
of $s_{\xi}$ by any bijection from $I_\xi$ onto $A_\xi\setminus W_\xi$.
There is such a bijection as both sets are infinite and countable by (i) and the properties of the sequence 
of $W_\xi$s.

Now suppose that \(\xi\) is a limit ordinal. To construct $s_\xi$
satisfying (ii) - (iv) fix a strictly increasing sequence of ordinals \(\{\xi_n\}_{n\in \N}\)
such that \(\xi_0 = 0\) and \(\sup_{n\in \N}{\xi_n} = \xi\).
Let \(\sigma: \N \to W_{\xi}\) be a bijection.
Recursively over \(n \geq 1\) we define a sequence {} \((t_n)_{n\in \N}\), where 
\begin{enumerate}
\item[(a)] \(t_n : \omega \cdot \xi_{n} \to W_{\xi} \) is an injection,
\item[(b)] \(t_n =^* s_{\xi_n}\), 
\item[(c)] \(t_k \subseteq t_n\) for \(k < n\),
\item[(d)] \(\sigma[n] \subseteq \ran(t_n)\).
\end{enumerate}
Put \(t_0 = \emptyset\). Having already defined \((t_k)_{k<n}\) satisfying (a) - (d)
we apply Lemma \ref{nice-ext} to $A=\omega \cdot \xi_{n-1}$, $f=t_{n-1}$,
$B=\omega \cdot \xi_{n}$, $g=s_{\xi_n}$, $C=W_\xi$, $F=\sigma[n]$.
To do so, we check that the hypotheses (1) - (7) of this lemma are satisfied:
hypotheses (1) and (7) follow directly from the definitions of $A, B, F, \sigma$; 
(2) follows from the inductive hypothesis (a); (3) from
the inductive hypothesis (ii) for $\xi_{n}<\xi$; (4) and (5) from 
(i) and the inductive hypothesis (iv) for $\xi_{n}<\xi$; and (6) from
the inductive hypothesis (iii) for $\xi_{n}<\xi$ and the inductive hypothesis (b).
So Lemma \ref{nice-ext} gives us $t_n$ as in (a) - (d). This completes
the recursive construction of $(t_n)_{n\in \N}$.   

Finally we set \(s_{\xi} := \bigcup_{n \in \N} t_n\) and check that it satisfies (ii) - (iv). 
As an increasing union of injections (by (a) and (c)) with the images contained 
in \(W_{\xi}\), \(s_{\xi}\) is an injection with the image in \(W_{\xi}\) itself,
so (ii) and one of the inclusions of (iv) hold.
The condition (d) ensures that \(\ran(s_{\xi}) = \bigcup_n \ran(t_n) \supseteq \bigcup_n \sigma[n] = 
\ran(\sigma) = W_{\xi}\), so the other inclusion of (iv) holds as well.
For any \(\alpha < \xi\) there is
\(n \in \N\) such that \(\xi_n > \alpha\) and 
\(s_{\xi}|{\omega \cdot \alpha} = 
t_{n}|{\omega \cdot \alpha} =^* s_{\xi_n}|{\omega \cdot \alpha} =^* s_{\alpha}\),
so $s_\xi$ satisfies (iii), as required.
\end{proof}

\begin{proposition}\label{b>omega1}
If  $\mathfrak b>\omega_1$, then
for every almost disjoint family $\A=\{A_\xi: \xi<\omega_1\}$ 
there is an isometric embedding $S:\ell_\infty^{c}(\omega_1)\rightarrow C(\N^*)$ satisfying
$$S(1_{\{\xi\}})=1_{[A_\xi]}$$
for every $\xi<\omega_1$.
\end{proposition}

\begin{proof}
Assume $\mathfrak b>\omega_1$ and fix an almost disjoint family
$\A=\{A_\xi: \xi<\omega_1\}$. We will construct an almost
increasing sequence $\{V_\alpha\}_{\alpha<\omega_1}$ of subsets of $\N$ such that
\begin{enumerate}
\item[(a)] $A_\xi\subseteq^* V_\alpha$ for $\xi<\alpha<\omega_1$,
\item[(b)] $A_\xi\cap V_\alpha=^*\emptyset$ for $\omega_1>\xi\geq \alpha$.
\end{enumerate}
This is done by induction on $\alpha<\omega_1$. For the successor case, we put
$V_{\alpha+1}=V_\alpha\cup A_\alpha$. For the limit case, we note that 
the inductive hypotheses (a), (b) imply that $B\cap C=^*\emptyset$
for any $B\in \B$ and any $C\in \CC$, where
$$\B=\{A_\xi: \xi<\alpha\}\cup\{V_\xi: \xi<\alpha\},\  \CC=\{A_\xi: \omega_1>\xi\geq \alpha\},$$
as well as $\B$ is countable and $|\CC|=\omega_1<\mathfrak b$.
By Theorem 3.3 from \cite{douwen} (case ${\mathfrak b}_7$) this implies
the existence of a separation $V_\alpha\subseteq \N$ such that the
appropriate cases of (a) and (b) hold.

Now use Lemma \ref{iso-chain} to conclude the proof.
\end{proof}

\section{Obstructions for isomorphic embeddings of $\ell_\infty^c(\kappa)$ into $\ell_\infty/c_0$}

\begin{lemma}\label{norm-outside}
Suppose that $S:\ell_\infty^c(\kappa)\rightarrow C(\N^*)$ is a
linear operator such that $S(1_{\{\xi\}})=1_{[A_\xi]}$, where
$\kappa$ is an uncountable cardinal and  $\{A_\xi: \xi<\kappa\}$ is an almost disjoint family. Let 
$$r= \inf_{C \in [\kappa]^{\leq \omega}}\sup_{f, x}\{|S(f)(x)|: f\in \ell_\infty^c(\kappa\setminus C),
\, \|f\|=1, \, x\in \bigcup_{\xi\in \kappa \setminus C}[A_\xi]\},$$
$$R=\inf_{C \in [\kappa]^{\leq \omega}}\sup_{f, x} \{|S(f)(x)|: 
f\in \ell_\infty^c(\kappa \setminus C)), \, \|f\|=1, \, x\not\in \bigcup_{\xi<\kappa}[A_\xi]\}.$$
Then $R\geq r$.
\end{lemma}

\begin{proof}
Fix $C\in [\kappa]^{\leq \omega}$. By the definition of \(r\), whenever we have a 
finite family \(\mc{F}\) of elements of \(\ell_{\infty}^c(\kappa)\) and \(\varepsilon > 0\), 
we can find \(g \in \ell_{\infty}^{c}(\kappa \setminus (C \cup \bigcup_{f \in \mc{F}} \supp f))\) 
of norm one such that \(|S(g)(x)| > r - \varepsilon\) for some \(x \in [A_\xi]\) where
\(\xi \in \kappa \setminus (C \cup \bigcup_{f \in \mc{F}} \supp f)\). Therefore, fixing $\varepsilon>0$
we choose a sequence $\{f_n: n\in \N\}$
of disjointly supported elements of norm one of $ \ell_\infty^c( \kappa \setminus C)$
such that for every \(n \in \N\) there is a different ordinal $\xi_n$ in 
$\kappa \setminus C$ and $x_n\in [A_{\xi_n}]$ with
$$|S(f_n)(x_n)|>r - \varepsilon.\leqno (*)$$ 
 For any \(x \in \N^*\) let \(\epsilon_k = \mathrm{sgn}(f_k(x))\). Note that
\[
\sum_{k \in \N} |S(f_k)(x)| = \sup_{E \in [\N]^{< \omega}} \sum_{k \in E} |S(f_k)(x)| = \sup_{E \in [\N]^{< \omega}} S\big(\sum_{k \in E} \epsilon_k f_k\big)(x)\]
\[\leq \sup_{E \in [\N]^{< \omega}}\norm{S} \cdot \big\|\sum_{k \in E} \epsilon_kf_k\big\| \leq \norm{S}.
\]
Applying Rosenthal's Lemma 1.1 of \cite{rosenthal} 
(for \(\Lambda = \Gamma = \N\), \(\mu_n(E) = \sum_{k \in E} |S(f_k)(x_n)|\) 
and \(E_n = \{n\}\))  and passing to an infinite subset of \(\N\) we may assume that 
for every $n\in \N$ we have 
$$\sum_{k\in \N\setminus\{n\}}|S(f_k)(x_n)|<\varepsilon.\leqno (**)$$
Let \(J: \ell_{\infty} \to \ell_{\infty}^{c}(\kappa)\) and {}\(R: C(\N^*) \to \ell_{\infty}\){}
be given by the formulas:
$$J((a_n)_{n\in \N}) = \sum_{n\in \mathbb{N}} a_n f_n$$
for \((a_n)_{n\in \N} \in \ell_{\infty}\), where the infinite sum is defined coordinatewise,
$$R(g) = (g(x_n))_{n\in \mathbb{N}}$$
for \(g \in C(\N^*)\). 
Let {} $T, T':\ell_\infty\rightarrow \ell_\infty$
be given by \(T = R \circ S \circ J\) and
$$T'((a_k)_{k\in\N}) = \Big(\sum_{k\in\N}a_kS(f_k)(x_n)\Big)_{n\in \N}.$$
By (**) $T'$ is well defined and bounded.
Then \(T|{c_0} = T'|{c_0}\) because $T(1_{\{k\}})=(S(f_k)(x_n))_{n\in \N}=T'(1_{\{k\}})$ for each $k\in \N$.
So by Theorem 2.5.4 from \cite{A&K} 
there exists an infinite subset \(M \subseteq \N\) such that \(T\) and \(T'\) coincide on \(\ell_{\infty}(M)\).

Let \(f = J(1_M) = \sum_{n\in M} f_n\), where the infinite sum is defined coordinatewise.
Observe that \(S(f)(x_n) =
T(1_{M})(n) = T'(1_M)(n) =  \sum_{k\in M}S(f_k)(x_n)\) for \(n \in M\). So by (*) and (**) for every $n\in M$ we have 
$$|S(f)(x_n)|\geq |S(f_n)(x_n)|-\sum_{k\in M\setminus\{n\}}|S(f_k)(x_n)|> r-2\varepsilon.\leqno (***)$$
Let $x$ be any accumulation point of $\{x_n: n\in M\}$ in {}$\N^*$. {}Clearly $|S(f)(x)|\geq r-2\varepsilon$ by (***).
As $x_n$s are contained in pairwise disjoint sets $[A_{\xi_n}]$ and each \([A_{\xi}]\) is open, we conclude
that $x\not\in  \bigcup_{\xi<\kappa}[A_\xi]$. Also, \(f \in \ell_\infty^c(\kappa \setminus C)\) and \(\norm{f} = 1\). Since \(C\) and $\varepsilon>0$ were arbitrary, we get \(R \geq r\).
\end{proof}

\begin{lemma}\label{mad} Let \(\mathcal{A}\) be a maximal almost disjoint family. 
Then for every subset \(X \subseteq \N\) such that 
\(X\) is not almost covered by finitely many elements of \(\mathcal{A}\) 
there exist \(\mathfrak{a}\) many \(A \in \mathcal{A}\) such that \(|A \cap X| = \omega\).
\end{lemma}

\begin{proof}
By contradiction, suppose that \(X \subseteq \N\) is not almost 
covered by finitely many elements of \(\mathcal{A}\), but 
\(|\{A \in \mathcal{A} : |A \cap X| = \omega\}| < \mathfrak{a}\). 
Let \(\mathcal{B} = \{A \in \mathcal{A} : |A \cap X| = \omega\}\).  
Then there exists \(C \subseteq X\) such that \(\{A\cap X : A \in \mathcal{B}\} \cup \{C\}\) 
is an almost disjoint family of subsets of \(X\). But this means 
that also the family \(\mathcal{A} \cup \{C\}\) is almost disjoint,
which contradicts the assumption that \(\mathcal{A}\) is maximal.
\end{proof}

\begin{proposition}\label{noextension} Suppose $\kappa$ is an uncountable cardinal and
$\A=\{A_\xi:\xi<\kappa\}$
is a maximal almost disjoint family\footnote{Note
that consistently $\kappa$ could have countable cofinality by the result of \cite{brendle}.} of subsets of $\N$.
Let $T$ be an isometric embedding $T:c_0(\kappa)\rightarrow C(\N^*)$ satisfying $T(1_{\{\xi\}})=1_{[A_\xi]}$
for every $\xi<\kappa$.
Then there is no linear operator $S:\ell_\infty^c(\kappa)\rightarrow C(\N^*)$
which extends $T$.{}
\end{proposition}

\begin{proof} 
Let $T:c_0(\kappa)\rightarrow C(\N^*)$ be such that $T(1_{\{\xi\}})=1_{[A_\xi]}$
for each $\xi<\kappa$. Suppose $S:\ell_\infty^c(\kappa)\rightarrow C(\N^*)$
is an operator which extends $T$. Let $R$ and $r$ be as in Lemma \ref{norm-outside}.
Fix $C \in [\kappa]^{\leq \omega}$.
Using the definition of $R$ find   $f\in \ell_\infty^c(\kappa \setminus C)$  with $\|f\|=1$ such that
$$|S(f)(x)|>R-1/2$$
for some  $x\not\in \bigcup_{\xi <\kappa}[A_\xi]$. By the continuity of $S(f)$ there is an infinite $B\subseteq \N$ such that
$x\in [B]$ and $|S(f)(y)|>R-1/2$ for every $y\in [B]$. As $x\not\in \bigcup_{\xi<\kappa}[A_\xi]$, we conclude
that $B$ is not almost covered by finitely many elements of $\mathcal{A}$, and so by Lemma \ref{mad} the set $B$
has infinite intersections with uncountably many elements
of $\mathcal{A}$.  Let $\xi \in \kappa \setminus C$ 
be such that $\xi$ is not in the support of $f$ and $A_\xi \cap B \ne^* \emptyset$. Then $||f\pm1_{\xi}\|=1$ but 
$$|S(f\pm1_{\xi})(y)|=|S(f)(y)\pm1_{[A_\xi]}(y)|=|S(f)(y)|+1>R-1/2+1>R+1/2$$
for every $y\in [A_\xi] \cap [B]$.  Since $C$ was arbitrary, this shows that $r\geq R+1/2$, which contradicts Lemma
\ref{norm-outside}.
\end{proof}

\section{Extending  isomorphisms  of isomorphic copies of $c_0(\kappa)$  to automorphisms  under 
$\mathsf{MA}_{\hbox{$\sigma$-{\small{\rm linked}}}}(\kappa)$}

In this section we consider an uncountable cardinal $\kappa < \cc$, assume
$\mathsf{MA}_{\hbox{$\sigma$-{\small{\rm linked}}}}(\kappa)$ and fix
two isomorphic copies $X_1, X_2$ of $c_0(\kappa)\subseteq \ell_\infty/c_0$
and an isomorphism $U: X_1\rightarrow X_2$. The aim of this section is to show
that under this hypothesis there is an automorphism $\tilde U$
of $\ell_\infty/c_0$ which extends $U$.

In the first part of this section we introduce terminology and notation concerning
objects related to $X_1, X_2$ and $U$, which will be used in the entire section.
Recall the terminology of  Definition \ref{def-pi} and Definition \ref{sim}. 
Fix some isomorphism $T: c_0(\kappa)\rightarrow X_1$. 
Then $S =U\circ T$ is an isomorphism $c_0(\kappa)\rightarrow X_2$.

In order to construct an extension of \(U\), we choose \(\mc{F} = \{f_{\xi} : \xi < \kappa\}\) 
and \(\mc{G} = \{g_{\xi} : \xi < \kappa\}\), two indexed subsets of \(\ell_{\infty}\), 
such that $\pi(f_\xi)=T(1_{\{\xi\}})$ and $\pi(g_\xi)=S(1_{\{\xi\}}) = U(\pi(f_{\xi}))$.
We will additionally assume that all the coordinates of all $f_\xi$s and $g_\xi$s
are rational numbers and \(\norm{f_{\xi}} = \norm{\pi(f_{\xi})}, \, \norm{g_{\xi}} = \norm{\pi(g_{\xi})}\). 
This can be easily achieved, as a perturbation of any $f\in \ell_\infty$ 
by an element of $c_0$ does
not affect the value of $\pi(f)$.  Observe 
that for any finite \(a \subseteq \kappa\) and \(\rho \geq \norm{T}, \norm{T^{-1}}, \norm{S}, 
\norm{S^{-1}}\) operators \(T, S\) witness that 
\[
\ang{\pi(f_\xi) : \xi \in a} \simi{\rho, \rho} \ang{1_{\{\xi\}} : \xi \in a} \simi{\rho, \rho} \ang{\pi(g_\xi) : \xi \in a},
\]
where \(1_{\{\xi\}}\)s are considered as elements of \(c_0(\kappa)\).

We will construct an isomorphism ${\bar U}: \ell_\infty\rightarrow \ell_\infty$ such that
${\bar U}[c_0] = c_0$ and ${\bar U}(f_\xi)-g_\xi\in c_0$ for each $\xi<\kappa$.
Then we can define
$\tilde U$ as 
$$\tilde U(\pi(f))=\pi( {\bar U}(f)).$$
The condition \(U[c_0] = c_0\) ensures that \(\tilde U\) is well-defined and injective. 
The surjectivity of \(\bar U\) implies the surjectivity of \(\tilde U\). The boundedness of \(\bar U\) provides the boundedness of \(\tilde U\).
Since \(\pi[\mc{F}]\) is linearly dense in \(X_1\) and \(\tilde U\) 
coincides with \(U\) on \(\pi[\mc{F}]\), \(\tilde U\) is indeed the desired extension of \(U\). The above terminology and notation will be used in the  rest of this section.

\begin{definition}\label{def-P} Suppose that $\kappa$ is an 
uncountable cardinal, \(X_1, X_2 \subseteq \ell_{\infty}/c_0\) two isomorphic copies of \(c_0(\kappa)\),
$T, S: c_0(\kappa)\rightarrow\ell_\infty/c_0$ and $ \mc{F}, \mc{G} \subseteq \ell_{\infty}/c_0$
as above. Let  \(\rho = \max\{\norm{T}, \norm{T^{-1}}, 
\norm{S}, \norm{S^{-1}}\}\). 
 
We define  $\PP =\PP_{ \mc{F}, \mc{G} }= \langle \PP, \leq_{\PP}\rangle$ as follows. 
The set $\PP$ consists of triples of the form $p = (n_p, M_p, a_p)$ where:
\begin{enumerate}
\item[(a)] $n_p \in \N$,
\item[(b)] $M_p$ is a $n_p \times n_p$ matrix with rational entries such that 
 $$\|v\|/c_2(8\rho^2) \leq \|M_{p}v\| \leq c_2(8\rho^2)\|v\|$$ 
for any vector $v \in \ell_\infty^{n_p}$ ($c_2$ as in Lemma \ref{bourgain2}),
\item[(c)] $a_p\in[\kappa]^{<\omega}$ such that for \(n \geq n_p\)

$$
\langle f_{\xi}|{{[n, \infty)}} : \xi \in a_p\rangle \simi{1,2} \langle\pi(f_{\xi}) :\xi \in a_p\rangle \text{ and }
 \langle \pi(g_{\xi}) :\xi \in a_p\rangle \simi{2,1} \langle g_{\xi}|{[n, \infty)} :\xi \in a_p\rangle.$$
\end{enumerate}

For two elements $p, q \in P$ we write $p \leq_{\PP} q$ if and only if:
\begin{enumerate}
\item[(i)] $n_p \geq n_q$,
\item[(ii)] the matrix $M_p$ is of the form
$$\begin{bmatrix}
M_q& 0 \\
0& M_{qp}
\end{bmatrix}$$
where $M_{qp}$ is a $ [n_q, n_p) \times [n_q, n_p)$ matrix,
\item[(iii)] $a_p \supseteq a_q$,
\item[(iv)] {} $M_{qp} (f_{\xi}| [n_q, n_p)) = g_{\xi}|[n_q, n_p)$ for every $\xi \in a_q$. {}
\end{enumerate}
\end{definition}

\begin{lemma}
$\PP_{\mc{F}, \mc{G}}= \langle P, \leq_{\PP}\rangle$ as above is a partial order.
\end{lemma}  

\begin{lemma}\label{amalgamation}
Suppose that $\kappa$ is an uncountable cardinal,  \(X_1, X_2 \subseteq \ell_{\infty}/c_0\)
two isomorphic copies of \(c_0(\kappa)\),
$T, S: c_0(\kappa)\rightarrow\ell_\infty/c_0$, \(\rho \in \R\) and $\mc{F}, \mc{G} 
\subseteq \ell_{\infty}/c_0$ as above. 
Let $\PP=\PP_{\mc{F}, \mc{G}}$.
Let $p = (n, H, a_p)\in \PP$ and $q = (n, H, a_q)\in\PP$, and let $N\in \N$. 
Then there exists $r = (n_r, M_r, a_r) \in \PP$ such that $r \leq p, q$ and $n_r \geq N$. 
\end{lemma}

\begin{proof}
Let $a_r=a_p\cup a_q$.  
By Lemmas \ref{lifting} and \ref{bounded-restriction} there is $n_r\in \N$ with $n, N<n_r$,
$|a_r|\leq c_1(8\rho^2)\sqrt{n_r-n}$
(the constant $c_1$ is as in Lemma \ref{bourgain2})
such that 
$$
\langle f_{\xi}|{[k, \infty)} : \xi \in a_r\rangle \simi{1,2} \langle\pi(f_{\xi}) 
:\xi \in a_r\rangle, \, \, \langle \pi(g_{\xi}) :\xi \in a_r\rangle \simi{2, 1} 
\langle g_{\xi}|{[k, \infty)} :\xi \in a_r\rangle \leqno {}(1){}
$$
for \(k \geq n_r\), and
\begin{align}\tag{2}\begin{split}
\langle f_\xi|[n, n_r) : \xi \in a_r \rangle & \overset{2, 1}{\sim} \langle f_\xi|[n, \infty) : \xi \in a_r \rangle, \\
\langle g_\xi|[n, \infty) : \xi \in a_r \rangle & \overset{1, 2}{\sim} \langle g_\xi|[n, n_r) : \xi \in a_r \rangle.
\end{split}\end{align} 

To complete the definition of $r$ we need to find $M_r$.  By Lemma \ref{sim-prop}, 
the condition (c) for \(p, q\) and the lower and upper bounds of \(S, T\) imply that for \(a \subseteq a_p\) or \(a \subseteq a_q\)
\[
 \langle f_{\xi}|{[n, \infty)} : \xi \in a\rangle \simi{\rho, 2\rho} \langle1_{\{\xi\}} 
 :\xi \in a\rangle \simi{2\rho, \rho} \langle g_{\xi}|{[n, \infty)} :\xi \in a\rangle,
\]
where \(1_{\{\xi\}}\)s are elements of \(c_0(\kappa)\). As a result, for any real-valued sequence \((t_\xi)_{\xi\in a_r}\)
\[
\Big\|\sum_{\xi \in a_r} t_{\xi} f_\xi|[n, \infty) \Big\| \leq \Big\|\, \sum_{\xi \in a_p} t_{\xi} f_\xi|[n, \infty)
 \, \Big\| + \Big\|\sum_{{\xi \in a_q \setminus a_p}} t_{\xi} f_\xi|[n, \infty)\, \Big\|\leq\]
\[\leq 2\rho\Big\|\sum_{\xi\in a_p} t_{\xi}1_{{\xi}} \, \Big\| + 
2 \rho \Big\|\sum_{\xi\in a_q \setminus a_p} t_{\xi}1_{{\xi}} \, 
\Big\| \leq 4 \rho \max_{\xi \in a_r} |t_{\xi}| = 4 \rho \Big\|\sum_{\xi\in a_r} t_{\xi}1_{{\xi}} \, \Big\|,
\]
and, on the other hand,
\[
\Big\|\sum_{\xi\in a_r} t_{\xi}1_{{\xi}} \, \Big\| \leq \rho \Big\|\sum_{\xi \in a_r} t_{\xi}
 \pi(f_\xi) \Big\| \leq \rho\Big\|\sum_{\xi \in a_r} t_{\xi} f_\xi|[n, \infty) \Big\|.
\]
We may also write analoguous inequalities for \(g_{\xi}\)s instead of \(f_{\xi}\)s. Consequently,
\begin{equation}\tag{3}
\langle f_\xi|[n, \infty) : \xi \in a_r \rangle \simi{\rho, 4\rho} \langle 1_{\{\xi\}} : \xi \in a_r \rangle \simi{4\rho, \rho}
\langle g_\xi|[n, \infty) : \xi \in a_r \rangle.
\end{equation}
Combining (2) and (3) by Lemma \ref{sim-prop} we get
\[
\langle f_\xi|[n, n_r): \xi \in a_r\rangle \simi{2\rho, 4\rho} \langle 1_{\{\xi\}} : \xi \in a_r \rangle \simi{4\rho, 2\rho} \langle g_\xi|[n, n_r) : \xi \in a_r\rangle. \leqno{(4)}
\]
Let \(Y_{\F} = \mr{span}\{f_\xi|{[n, n_r)}: \xi \in a_r\}\), \(Y_{\G} =  \mr{span}\{g_\xi|{[n, n_r)}: \xi \in a_r\}\) and \(V: Y_{\F} \to Y_{\G}\) be given by \(V(f_\xi|{[n, n_r)}) = g_{\xi}|{[n, n_r)}\). Then \(\dim(Y_{\F}) = \dim(Y_\G) = |a_r|\) and (4) implies that \(d(Y_\F, \ell_\infty^{|a_r|}), d(Y_\G, \ell_\infty^{|a_r|}) \leq 8\rho^2\) and \(\norm{V}, \norm{V^{-1}} \leq 8\rho^2\). 

So, as $|a_r|\leq c_1(8\rho^2)\sqrt{n_r-n}$, we can apply Lemma \ref{bourgain2} to conclude that there is
an automorphism $W$ of $\ell_\infty^{n_r\setminus n}$ which extends $V$ such that
$$\|W\|, \|W^{-1}\|\leq c_2(8\rho^2)\leqno {}(5){}$$ 
and, since \(f_\xi\)s and \(g_\xi\)s have all coordinates rational, the matrix $M$ of $W$ has rational entries. We define:
\[
M_r = \begin{bmatrix}
    H & 0\\
    0 & M
\end{bmatrix}.
\]
Now we are left with checking that $r\in \PP$ and $r\leq p, q$. Condition
(a) of Definition \ref{def-P} is clear. Also $M_r$
is a \(n_p \times n_p\) matrix with rational entries by (b) for $p, q$ and the choice of $M$.
 Moreover, for any \(v \in  \ell_\infty^{n_r}\) 
\[
M_rv = H(v|[0, n))\cup M(v|[n, n_r))
\] 
so \(\norm{M_rv} = \max\{\norm{H(v|{[0, n))}}, \norm{M(v|{[n, n_r))}}\}\). 
Since by (b) for $p, q$ and by (5) we have 
\(\norm{H}, \norm{H^{-1}}, \norm{M}, \norm{M^{-1}} \leq c_2(8\rho^2)\), we get that
$$
\frac{\|v\|}{c_2(8\rho^2)} = \frac{\max \{\norm{v|{[0, n)}}, \norm{v|{[n, n_r)}}\}}{c_2(8\rho^2)}
\leq \max\{\norm{Hv|{[0, n)}}, \norm{Mv|{[n, n_r)}}\} \leq $$
$$ \leq c_2(8\rho^2) \cdot\max\{\norm{v|{[0, n)}}, \norm{v|{[n, n_r)}}\} = c_2(8\rho^2)\|v\|.
$$

So the condition (b) from Definition \ref{def-P} is satisfied for
$r = (n_r, M_r, a_r)$. Also (c) of  Definition \ref{def-P} is satisfied for $r$ by (1)
and so $r\in \PP_{\mc{F}, \mc{G}}$. To check that \(r \leq p, q\) first we note that the conditions (i) - (iii) 
of Definition \ref{def-P} follow directly from the choices of $n_r, M_r$ and $a_r$ respectively.
Finally condition (iv) follows from the definition of \(V\) and the fact that $W$ extends $V$.
\end{proof}

\begin{lemma}\label{ccc}
$\PP_{\mc{F}, \mc{G}}$ as in Definition \ref{def-P} is $\sigma$-linked
and so it satisfies the c.c.c.
\end{lemma}

\begin{proof}
For $n\in \N$ and for an $n\times n$ matrix $H$  with rational entries let
$$\PP(n, H)=\{p\in \PP_{\F, \mc{G}}: n=n_p, H=M_p\}.$$
It is clear that there are countably many $n, H$ as above
and that every element of $\PP_{\F, \mc{G}}$
belongs to $\PP(n, H)$ for some $n, H$ as above. Lemma \ref{amalgamation} implies
that each $\PP(n, H)$ is linked, i.e., every two elements of it are compatible.      
\end{proof}
      
\begin{lemma}\label{D-dense}
For every $n\in \N$ the set $$D_n=\{p\in \PP_{\F, \G}: n_p\geq n\}$$
is dense in $\PP_{\F, \G}$ {}  as in Definition \ref{def-P}.
\end{lemma}  

\begin{proof}
Apply Lemma \ref{amalgamation} to $p=q$ and $N\geq n$.
\end{proof}    
   
\begin{lemma}\label{E-dense}
For every $\xi<\kappa$ the set
$$E_\xi=\{p\in \PP_{\F, \G}: \xi\in a_p\}$$
is dense in $\PP_{\F, \G}$  as in Definition \ref{def-P}.
\end{lemma}

\begin{proof} Fix $\xi\in \kappa$ and $p = (n_p, M_p, a_p)$. 
Since \(\norm{f_{\xi}} = \norm{\pi(f_{\xi})}\) and \(\norm{g_{\xi}} = \norm{\pi(g_\xi)}\) by the choice of 
$f_\xi$ and $g_\xi$ at the beginning of this section, the  triple \(q = (n_p, M_p, \{\xi\})\)
belongs to \(\PP_{\F, \G}\) as condition (c) of Definition \ref{def-P} is satisfied for $q$.
Apply Lemma \ref{D-dense} to find $r \leq p, q$. Then \(r \in E_{\xi}\).
\end{proof} 

\begin{proposition}\label{ma} Suppose that $\kappa$ is an infinite cardinal.
Assume $\mathsf{MA}_{\hbox{$\sigma$-{\small{\rm linked}}}}(\kappa)$. 
Let $X_1, X_2$ be two isomorphic copies of $c_0(\kappa)$ inside $\ell_\infty/c_0$ and let
$U: X_1\rightarrow X_2$ be an isomorphism. Then
there is an automorphism $\tilde U$ of $\ell_\infty/c_0$ which extends $U$. 
\end{proposition}

\begin{proof} As noted at the beginning of this section, under the hypothesis of the proposition
we have certain isomorphic embeddings $S, T$ of $c_0(\kappa)$ into $\ell_\infty/c_0$ and
$f_\xi, g_\xi\in \ell_\infty$ for $\xi<\kappa$ such that $\pi(f_\xi)=T(1_{\{\xi\}})$,
$\pi(g_\xi)=S(1_{\{\xi\}})$, and it is enough
to construct an automorphism ${\bar U}: \ell_\infty\rightarrow \ell_\infty$ such that
${\bar U}[c_0] = c_0$  and ${\bar U}(f_\xi)-g_\xi\in c_0$ for each $\xi<\kappa$.

To do so, we consider the partial order $\PP=\PP_{\F, \G}$. 
By Lemma \ref{ccc} and Lemmas \ref{D-dense} and \ref{E-dense} and the hypothesis of
$\mathsf{MA}_{\hbox{$\sigma$-{\small{\rm linked}}}}(\kappa)$ there is a filter $\mathbb G\subseteq \PP$
such that $D_n\cap\mathbb G\not=\emptyset\not=E_\xi\cap \mathbb G$ for each $n\in \N$ and each $\xi<\kappa$. 

Now define an $\N\times\N$ matrix $M$ by putting $M(i, j)=k$ if and only if
there is $p\in \mathbb G$ such that $M_p(i, j)=k$. Note that since $\mathbb G$ is a filter, given
$p, q\in \mathbb G$ such that $i, j<n_p, n_q$, there is $r\in \mathbb G$ such that $r\leq p, q$
and so $M_p(i, j)=M_r(i, j)=M_q(i, j)$. By Lemma \ref{D-dense}, for any pair 
\((i,j) \in \N\) there is \(p \in \mathbb{G}\) such that \(n_p > i, j\) and
 \(M_p\) is an \(n_p \times n_p\) matrix. Hence, $M$ is a well defined $\N\times \N$ matrix. 

Moreover, according to
Lemma \ref{D-dense}, there is a decreasing sequence $\{p_k: k\in \N\}\subseteq \mathbb G$
with $p_0$ being the trivial condition
such that $n_{p_k}\geq k$. It follows from condition (ii) of Definition \ref{def-P}
that there are $[n_{p_k}, n_{p_{k+1}})\times [n_{p_k}, n_{p_{k+1}})$ matrices $M_k$ such that $M$ 
is block diagonal  built from blocks 
$M_k$, that is $M(i, j)=0$ unless there is $k$ such that $i, j\in [n_{p_k}, n_{p_{k+1}})$ and
$M(i, j)=M_k(i, j)$. 
In particular $M 1_{\{i\}}$, which is one of the columns of
the matrix $M$, is a finitely supported vector for any \(i \in \N^*\). It follows that $M$ defines an operator 
${\bar U_{00}}: c_{00}\rightarrow c_{00}$  given by ${\bar U}_{00}(f)=Mf$
for $f\in c_{00}$. 

Condition (b) of Definition \ref{def-P} implies that $\|M_k\|\leq c_2(8\rho^2)$ for each $k$ and so
$\|{\bar U}_{00}\|\leq c_2(8\rho^2)$.  Consequently, \(\bar U_{00}\) may be extended to an 
operator ${\bar U}_{0}$ on \(c_0\) with the same norm.  Moreover, condition (b) of Definition \ref{def-P} implies also that 
for each $k\in \N$ there are $[n_{p_k}, n_{p_{k+1}})\times [n_{p_k}, n_{p_{k+1}})$ matrices $\tilde M_k$
which are inverses of $M_k$s. The same argument as for $M$ implies that
the matrix $\tilde M$ formed by the matrices $\tilde M_k$ defines the inverse 
of ${\bar U}_0$ of the norm not exceeding $c_2(8\rho^2)$.

Multiplications by matrices $M$ and $\tilde M$ define 
operators ${\bar U}$, ${\bar U}^{-1}$ on $\ell_\infty$ (Propositions 2.1 and 2.2  of \cite{cristobal})
whose composition (from either side) is the identity. In fact ${\bar U}={\bar U}_0^{**}$ and similarly for the inverse
(Proposition 2.5 of \cite{cristobal}).
It follows that ${\bar U}$ given by the matrix $M$ is an
automorphism of $\ell_\infty$ with \(\bar U [c_0] = c_0\).

Now fix $\xi<\kappa$ and applying Lemma \ref{E-dense} fix $p\in \mathbb G$ such that
$\xi\in a_p$. We claim that ${\bar U}(f_\xi) =g_\xi +h_{\xi}$,
where $h_{\xi} \in \ell_\infty(n_p)$. Indeed, since $\mathbb G$ is a filter, 
$M(i, j)=k$ if and only if there is $q\in \mathbb G$ such that $M_q(i, j)=k$ and $q\leq p$,
so the claim follows from conditions (ii) and (iv) of Definition \ref{def-P}.
\end{proof}

\section{Questions}
     
Besides Problem 12 of \cite{book} whether $\ell_\infty/c_0$ can be consistently universally $\kappa$-injective
for $\kappa>\omega_1$ we believe the following problems are open, natural and interesting:
\begin{questions}\label{questions1}
Is any of the following statements provable in {\sf ZFC} (or its negation is consistent with {\sf ZFC})
\begin{enumerate}
\item if $\ell_\infty/c_0$ is  $2^\omega$-automorphic, then $2^\omega=\omega_1$;
\item if $\ell_\infty/c_0$ is $2^\omega$-extensible, then $2^\omega=\omega_1$;
\item $\ell_\infty/c_0$ is not $c_0(2^\omega)$-automorphic;
\item $\ell_\infty/c_0$ is not $c_0(2^\omega)$-extensible;
\item $\ell_\infty/c_0$ is not $c_0(\mathfrak a)$-automorphic;
\item $\ell_\infty/c_0$ is not $c_0(\mathfrak a)$-extensible.
\end{enumerate}
\end{questions}
We note that $\ell_\infty/c_0$ is not automorphic.
This is because $\ell_\infty/c_0$ has  a decomposition as $\ell_\infty/c_0\oplus \ell_\infty$
and it is proved in \cite{auto-linros} that $\ell_\infty$ is not automorphic.
Taking into account Proposition \ref{noextension}, one way to solve in the negative items (5) and (6) (and so
(3) and (4)) may be through 
generalizing Proposition \ref{embedding} from $\omega_1$ to $\mathfrak a$ by constructing an isomorphic embedding
of $\ell_\infty^c(\mathfrak a)$ into $\ell_\infty/c_0$ in {\sf ZFC}.
We note here that this may be delicate because, for example, it is proved
in \cite{christina} that consistently $2^\omega$ is arbitrarily big and 
$\ell_\infty(c_0(\omega_2))$ cannot be isomorphically embedded into $\ell_\infty/c_0$.
So in such a situation $\ell_\infty^c(\omega_2)$ simply does
not embed into $\ell_\infty/c_0$ as $\ell_\infty(c_0(\omega_2))\subseteq \ell_\infty^c(\omega_2)$. 
However, in these (Cohen) models $\mathfrak a=\omega_1$. 

Another observation is that in the model of  Theorem \ref{main-inj} all the items besides (1) and (2) of the
above question hold by Proposition \ref{noextension}, as $\mathfrak a=2^\omega$
and $\ell_\infty^c(\mathfrak a)$ embeds into $\ell_\infty/c_0$ since the Boolean algebra 
of countable and cocountable subsets of $2^\omega$ has cardinality $2^\omega$. 
Also we do not know the answer to the following
\begin{question} Suppose that $\kappa<\mathfrak b$ is an uncountable cardinal.
Does every isometric embedding of $T: c_0(\kappa)\rightarrow\ell_\infty/c_0$
extend to an isometric embedding $T: \ell_\infty^c(\kappa)\rightarrow\ell_\infty/c_0$?
\end{question}

The positive answer would generalize Theorem \ref{b-equiv}. In fact some of the
questions above concerning the existence of isomorphic embeddings may be solved by
providing even stronger Boolean embeddings as we do not know the answers to the following: 

\begin{questions} Are the following provable in {\sf ZFC}:
\begin{enumerate}
\item The Boolean algebra of all countable and all cocountable subsets of $\mathfrak a$
embeds into $\wp(\N)/Fin$;
\item Every embedding of the Boolean algebra $Fincofin(\kappa)$ into $\wp(\N)/Fin$
can be extended to an embedding of the Boolean algebra of all countable and all cocountable subsets of 
$\kappa$ if $\kappa<\mathfrak b$.
\end{enumerate}
\end{questions}

\bibliographystyle{amsplain}

\end{document}